\documentclass[12pt]{article}

\usepackage[T1]{fontenc}
\usepackage{lmodern}

\usepackage{amsmath, amssymb, amsthm}
\usepackage[margin=2.5cm]{geometry}
\usepackage{color, graphicx}
\usepackage{multirow}
\usepackage{parskip}

\usepackage{etoolbox}
\makeatletter
\patchcmd{\@maketitle}{\LARGE \@title}{\LARGE\bfseries\@title}{}{}

\renewcommand{\@seccntformat}[1]{\csname the#1\endcsname.\quad}
\makeatother

\definecolor{darkblue}{rgb}{0,0,.5}

\usepackage{hyperref}
\hypersetup{
	colorlinks=true,		
	linkcolor=darkblue,		
	citecolor=darkblue,		
	urlcolor=darkblue 		
}

\makeatletter
\def\th@plain{%
	\thm@notefont{}
	\itshape 
}
\def\th@definition{%
	\thm@notefont{}
	\normalfont 
}

\renewenvironment{proof}[1][\proofname]{\par
	\normalfont
	\topsep0\p@\@plus3\p@ \trivlist
	\item[\hskip\labelsep\itshape
	#1\@addpunct{.}]\ignorespaces
}{%
	\qed\endtrivlist
}
\makeatother

\usepackage[capitalize, nameinlink, noabbrev]{cleveref}

\newtheorem{theorem}{Theorem}[section]

\newtheorem{corollary}[theorem]{Corollary}
\newtheorem{proposition}[theorem]{Proposition}
\theoremstyle{definition}
\newtheorem{definition}[theorem]{Definition}

\newtheorem{remark}[theorem]{Remark}

\renewcommand\theenumi{(\roman{enumi})}

\usepackage[shortlabels]{enumitem}

\usepackage{empheq}
\definecolor{myblue}{rgb}{.8, .8, 1}

\usepackage[numbers,sort&compress]{natbib}

\bibpunct[, ]{[}{]}{,}{n}{,}{,}
\usepackage[numbers,sort&compress]{natbib}
\bibpunct[, ]{[}{]}{,}{n}{,}{,}

\def\<{\langle}
\def\>{\rangle}

\begin{document}

\title{Composite Lyapunov Criteria for Stability and Convergence with Applications to Optimization Dynamics}

\author{
Hassan Saoud\thanks{Department of Mathematics and Natural Sciences \& Center of Applied Mathematics and Bioinformatics (CAMB), Gulf University for Science and Technology, P.O. Box 7207, Hawally 32093, Kuwait. Email: \texttt{saoud.h@gust.edu.kw}.}
}

\date{}

\maketitle
\begin{abstract} We propose a composite Lyapunov framework for nonlinear autonomous systems that ensures strict decay through a pair of differential inequalities. The approach yields integral estimates, quantitative convergence rates, vanishing of dissipation measures, convergence to a critical set, and semistability under mild conditions, without relying on invariance principles or compactness assumptions. The framework unifies convergence to points and sets and is illustrated through applications to inertial gradient systems and Primal--Dual gradient flows.
\end{abstract}

{\small
\noindent{\bfseries Keywords:} Composite Lyapunov function, strict decay inequality, stability of sets, pointwise asymptotic stability, semistability, optimization dynamics.

\noindent{\bfseries AMS Subject Classifications:} 34D05, 34D20, 93D05, 37B25, 49J40
}
\section{Introduction and Problem Setting}
\label{sec:intro-framework}
The analysis of asymptotic behavior in nonlinear dynamical systems remains a central topic in control theory and applied mathematics. Classical tools such as Lyapunov’s direct method, LaSalle’s invariance principle, and Matrosov-type criteria have long provided the foundation for stability analysis across a broad range of systems, including mechanical, optimization, and distributed dynamics. Despite their success, these techniques exhibit structural limitations. LaSalle’s invariance principle identifies invariant sets but does not, by itself, ensure convergence to them, while Matrosov-type conditions typically yield weak asymptotic results, often expressed in terms of the vanishing of dissipation measures in the limit inferior sense. Strong convergence properties usually require additional compactness or stability assumptions on the invariant set.

\medskip
\noindent
The proposed approach draws its inspiration from the classical progression of asymptotic stability theory—from LaSalle’s invariance principle, through Matrosov’s auxiliary-function method, to later reinterpretations such as the LaSalle-type formulation in~\cite{AstolfiPraly2011}.
Unlike these methods, it relies on a single composite Lyapunov construction that achieves strict decay directly, without invoking nested or auxiliary functions. 
LaSalle’s invariance principle (see, e.g., \cite{khalil}) provides a foundational result: when a Lyapunov function is nonincreasing along trajectories, convergence is guaranteed only to the largest invariant set where its derivative vanishes; this is qualitative and often relies on compactness. 

\medskip
\noindent
To overcome these limitations, Matrosov~\cite{Matrosov1962} introduced an auxiliary function whose derivative is strictly negative on the set where the derivative of the Lyapunov function vanishes, thus establishing asymptotic stability without requiring compactness. 
This foundational idea was subsequently developed by Loria, Mazenc, and Teel~\cite{Loria2005}, who proposed a \emph{nested Matrosov theorem} to handle multiple auxiliary functions and nonautonomous systems, ensuring uniform convergence under persistency-of-excitation conditions. 
Further generalizations were presented by Mazenc and Dragan Ne\v{s}i\'c~\cite{Mazenc2007} for time-varying systems and by Teel et al.~\cite{Teel2016} for \emph{differential inclusions}, where tailored Matrosov functions were used to establish attractivity of compact sets under arbitrary switching among finitely many vector fields. 
Within the classical ODE setting, Astolfi and Praly~\cite{AstolfiPraly2011} formulated a \emph{LaSalle-type version} of Matrosov’s theorem, extending the analysis to systems with multiple equilibria and connecting Matrosov reasoning with the invariance principle through a family of auxiliary functions and weak decay conditions.

\medskip
\noindent
This work develops a constructive alternative that derives strict decay directly from a pair of differential inequalities. Consider the autonomous system 
\begin{equation}
    \label{eq:prob}
    \dot{x}(t) = f(x(t)), \qquad x(0) = x_0 \in \mathbb{R}^n,
\end{equation}
where $f:\mathbb{R}^n \to \mathbb{R}^n$ is locally Lipschitz. 
Under this assumption, the system admits a unique solution $x(\cdot)$ defined on a maximal interval of existence. 
When required, we assume that the system is forward complete so that every trajectory is defined for all $t \ge 0$.
 The equilibrium set is denoted
\[
E_f := \{x \in \mathbb{R}^n : f(x)=0\}.
\]

A point $x^\ast \in E_f$ is said to be \emph{Lyapunov stable} if, for every neighborhood $U$ of $x^\ast$, there exists a smaller neighborhood $V \subset U$ such that all trajectories starting in $V$ remain in $U$ for all future times. If, in addition, $x(t)\to x^\ast$ as $t\to\infty$, the equilibrium is \emph{asymptotically stable}.  

\medskip
\noindent
In this paper, the analysis extends beyond isolated equilibria to the \emph{stability of sets} associated with the system~\eqref{eq:prob}, with a particular focus on \emph{pointwise asymptotic stability of a set (PAS)}. 
Pointwise asymptotic stability means that every point in the set is Lyapunov stable and that every trajectory starting near the set converges to a limit within it. 

\medskip
\noindent
To clarify this notion, it is important to note that in the literature, PAS is typically defined with respect to the set of equilibria (see e.g., \cite{goebel10, goebel16}), while several works use the term \emph{semistability} to describe the same property when restricted to equilibria (see, e.g., \cite{bhat99,hui09}). 
In this paper, the term PAS is used when referring to an arbitrary set, whereas the term \emph{semistability} is reserved for the case where the set coincides with the equilibrium set. 
An equilibrium is thus \emph{semistable} if it is Lyapunov stable and every trajectory starting near it converges to a (possibly different) Lyapunov stable equilibrium. 
When the equilibrium set reduces to a single point, semistability and asymptotic stability coincide.

\medskip
\noindent
The concepts of pointwise asymptotic stability and semistability provide a natural framework for systems with nonisolated equilibria and have been investigated in several settings, including differential equations, differential inclusions, and hybrid systems (see, for instance, \cite{bhat99, hui09, goebel10, goebel16, goebel18, saoud15, SaoudTheraDao2025}). These studies form the conceptual background for the results presented in this paper.

\medskip
\noindent
Having established the stability notions of interest, we now turn to the analytical framework that allows their verification. 
The proposed approach constructs a composite Lyapunov function based on two differential inequalities that encode the system’s dissipative behavior. 
To this end, we introduce two continuously differentiable functions $V_1,V_2:\mathbb{R}^n\to\mathbb{R}$ and two continuous, nonnegative functions $N_1,N_2:\mathbb{R}^n\to[0,\infty)$. 
The functions $V_1$ and $V_2$ serve as Lyapunov candidates, while $N_1$ and $N_2$ quantify the dissipation revealed by their derivatives. Along the trajectories of \eqref{eq:prob}, these functions satisfy
\[
\dot V_1(x) \le -N_1(x), 
\qquad 
\dot V_2(x) \le -N_2(x) + h(N_1(x)),
\]
where $h:[0,\infty)\to[0,\infty)$ is continuous with $h(0)=0$.

\medskip
\noindent
The quantities $N_1$ and $N_2$ are called \emph{observables} because they measure the rate at which dissipation occurs in the system and indicate how far the trajectory is from the asymptotic regime. They can be viewed as measurable indicators of decay: each $N_i(x)$ represents a nonnegative quantity that vanishes exactly when the corresponding Lyapunov function $V_i$ ceases to decrease. In this sense, $N_i$ “observes’’ how much the system is dissipating energy at state $x$. Their role is twofold. Analytically, they capture the amount of dissipation appearing in the inequalities that govern the time derivatives of $V_1$ and $V_2$, providing a direct handle on integral estimates and asymptotic convergence. Geometrically, their joint vanishing defines the long-term behavior of the system. The set
\[
E := \{x \in \mathbb{R}^n : N_1(x)=0,\; N_2(x)=0\}
\]
thus identifies the asymptotic regime reached by trajectories, acting as an observable-based description of equilibria. This viewpoint is particularly convenient in nonsmooth or distributed settings where the vector field $f$ may not be explicitly available but the dissipation measures $N_1$ and $N_2$ remain well defined.

\medskip
\noindent
Since $N_1$ and $N_2$ encode complementary information about system dissipation, combining them into a single Lyapunov expression is a natural step toward capturing the global decay structure of the dynamics.  

A small-gain condition on $h$ ensures that the influence of $N_1$ on the decay of $V_2$ does not cancel the overall dissipation. 
To capture their combined effect, we introduce the \emph{composite Lyapunov function}
\[
W(x) = V_1(x) + \delta V_2(x), \qquad \delta>0,
\]
and show that, for a suitable $\delta$, there exists $\gamma>0$ such that
\[
\dot W(x) \le -\gamma\big(N_1(x)+N_2(x)\big).
\]
This strict decay inequality aggregates the information from both differential inequalities, enforcing the vanishing of $N_1(x(t))$ and $N_2(x(t))$ along every trajectory. 
It forms the cornerstone of the analysis developed in this paper and is reminiscent of Matrosov’s reasoning—where one function compensates for another—yet it yields a direct decay relation that bypasses hierarchical auxiliary-function constructions and successive negativity checks.

\medskip
\noindent
Every bounded trajectory satisfies $\mathrm{dist}(x(t),E)\to 0$ as $t\to\infty$, meaning trajectories asymptotically approach the critical set where both dissipation measures vanish. Since equilibria force $N_1$ and $N_2$ to be zero, the equilibrium set $E_f$ is contained in $E$. In many cases, the two sets coincide, so trajectories approach the equilibrium set, leading to semistability when each equilibrium is Lyapunov stable. More generally, if every point of $E$ is Lyapunov stable, then $E$ is pointwise asymptotically stable (PAS). The strict decay inequality also provides integral bounds and convergence of observables, which form the foundation for these stability results. Finally, when $E$ reduces to a single equilibrium, global asymptotic stability follows.

\medskip
\noindent
This composite Lyapunov construction unifies and extends classical Lyapunov and Matrosov methods. 
It establishes convergence and stability properties without relying on compactness or invariance principles and applies naturally to nonlinear and optimization-driven systems. 
In particular, it captures the dynamics of inertial gradient–like systems and Primal-–Dual gradient flows that arise in constrained optimization and networked control problems. 
These two classes of systems will serve as detailed case studies in the sequel, illustrating how the general results of the paper ensure convergence and stability of their equilibrium sets. 
Beyond optimization, the same framework can describe dynamic adjustment processes in economics, such as signaling-based Cournot competition~\cite{Daher2012}, where the evolution of agents’ strategies toward equilibrium can be interpreted through stability and convergence analysis.

\medskip
\noindent
In addition to qualitative convergence, the approach provides quantitative estimates on the rate at which trajectories approach the critical set, linking decay inequalities with integral bounds and error–bound conditions.
These results make it possible to evaluate how quickly the system stabilizes, offering explicit performance guarantees that connect Lyapunov decay with convergence speed—an aspect relevant to optimization, control, and economic dynamics.
\section{Main Results}
\label{sec:main}
The following results establish the complete framework built on the pair of differential 
inequalities introduced earlier. The first step (Subsection~\ref{subsec:strict-decay}) 
proves that a suitable combination of $V_1$ and $V_2$ yields a composite Lyapunov function 
that decreases strictly along trajectories. Subsequent subsections derive its main 
consequences—integral bounds, convergence to the critical set, quantitative decay rates, 
and stability properties. Together, they provide a unified and constructive approach to 
analyzing asymptotic behavior without relying on invariance principles or compactness 
assumptions.
\subsection{Strict Decay of the Composite Function}
\label{subsec:strict-decay}
We begin with the fundamental decay result that underlies the entire framework.
\begin{theorem}[\textbf{Strict Decay of a Composite Lyapunov Function}]
\label{thm:strict-decay}
Suppose there exist continuously differentiable functions 
$V_1,V_2 : \mathbb{R}^n \to \mathbb{R}$, 
continuous nonnegative functions $N_1,N_2 : \mathbb{R}^n \to [0,+\infty)$, 
and a continuous function $h : [0,+\infty) \to [0,+\infty)$ with $h(0)=0$ 
such that, along every solution,
\[
\dot{V}_1(x) \le -N_1(x), \qquad 
\dot{V}_2(x) \le -N_2(x) + h(N_1(x)).
\]
Assume in addition that
\[
L := \sup_{r>0}\frac{h(r)}{r} < +\infty.
\]
Then, for any $\delta \in (0,1/L)$ and $W(x):=V_1(x)+\delta V_2(x)$, 
there exists a constant
\[
\gamma := \min\{\,1-\delta L,\ \delta\,\} > 0
\]
such that, along all solutions,
\[
\dot{W}(x) \le -\gamma \big(N_1(x)+N_2(x)\big).
\]
\end{theorem}
\begin{proof}
Take any solution $x(\cdot)$ of the system. By assumption, the functions $V_1$ and $V_2$ satisfy
\[
\dot V_1(x) \le -N_1(x), 
\qquad 
\dot V_2(x) \le -N_2(x) + h(N_1(x)).
\]
Fix a parameter $\delta>0$ and define the composite Lyapunov function
\[
W(x) := V_1(x) + \delta V_2(x).
\]
Differentiating along the trajectory gives
\[
\dot W(x) = \dot V_1(x) + \delta \dot V_2(x) 
\le -N_1(x) - \delta N_2(x) + \delta\,h(N_1(x)).
\]
Now apply the global slope bound. By definition of
\[
L := \sup_{r>0} \frac{h(r)}{r},
\]
we have $h(s) \le L s$ for every $s \ge 0$. In particular, for $s = N_1(x)$,
\[
\delta h(N_1(x)) \le \delta L N_1(x).
\]
Substituting this into the inequality for $\dot W$ yields
\[
\dot W(x) \le -(1-\delta L)\,N_1(x) - \delta N_2(x).
\]
Choose any $\delta \in (0,1/L)$. Then both coefficients $1-\delta L$ and $\delta$ are positive. 
Since $N_1(x),N_2(x)\ge 0$, we can estimate the right-hand side by pulling out the smaller of the two coefficients:
\[
(1-\delta L)N_1(x) + \delta N_2(x) 
\ge \min\{\,1-\delta L,\;\delta\,\}\,\big(N_1(x)+N_2(x)\big).
\]
Therefore,
\[
\dot W(x) \le -\gamma \big(N_1(x)+N_2(x)\big),
\qquad 
\gamma := \min\{1-\delta L,\;\delta\} > 0.
\]
This proves that the composite Lyapunov function decreases strictly along every solution, 
with a decay rate controlled by $\gamma$.
\end{proof}
\noindent
The constants in the strict decay inequality can in fact be optimized, 
leading to the following refinement.
\begin{remark}[Optimal choice of constants]
The decay estimate in Theorem~\ref{thm:strict-decay} holds for any 
$\delta \in (0,1/L)$ with 
\[
\dot W(x) \le -\gamma \big(N_1(x)+N_2(x)\big), 
\qquad 
\gamma = \min\{\,1-\delta L,\;\delta\,\}.
\]
The guaranteed rate $\gamma$ depends on the tuning parameter $\delta$. 
Since $1-\delta L$ decreases linearly with $\delta$ while $\delta$ increases linearly, 
the quantity $\gamma$ is maximized when the two terms coincide, i.e.,
\[
1-\delta L = \delta.
\]
Solving for $\delta$ gives the optimal choice
\[
\delta^\ast = \frac{1}{1+L}.
\]
Substituting this value into the expression for $\gamma$ yields
\[
\gamma^\ast = \frac{1}{1+L}.
\]
Thus, the sharpest form of the strict decay inequality is
\[
\dot W(x) \le -\frac{1}{1+L}\,\big(N_1(x)+N_2(x)\big).
\]
\end{remark}
\subsection{Local Strict Decay under Bounded Trajectories}
The global result in Theorem~\ref{thm:strict-decay} relies on the uniform slope bound 
$L = \sup_{r>0} h(r)/r < \infty$, which guarantees a global decay rate valid for all trajectories. 
When such a bound is available only on a bounded region or along a specific trajectory, 
a local version can be established with trajectory-dependent constants.
The constant $L$ in Theorem~\ref{thm:strict-decay} controls the interaction term $h(N_1)$ globally. 
In many systems, however, $h(r)/r$ may be bounded only on the range of $N_1$ actually visited by the trajectory. 
To capture this situation, we restrict attention to a 
\emph{positively-invariant set} $\Omega \subseteq \mathbb{R}^n$ that contains the entire 
trajectory $\{x(t): t \ge 0\}$. On such a set, the ratio $h(s)/s$ remains finite, 
allowing a local version of the strict-decay inequality with constants that depend 
on $\Omega$ and ultimately on the initial condition.
\begin{theorem}[\textbf{Local Strict Decay with Optimal Constants}]
\label{thm:strict-decay-local}
Suppose there exist $V_1,V_2\in C^1(\mathbb{R}^n)$, continuous functions 
$N_1,N_2:\mathbb{R}^n\to[0,\infty)$, and a continuous function $h:[0,\infty)\to[0,\infty)$ with $h(0)=0$
such that, along every trajectory,
\[
\dot V_1(x)\le -N_1(x),\qquad
\dot V_2(x)\le -N_2(x)+h\big(N_1(x)\big).
\]
Let $\Omega\subseteq\mathbb{R}^n$ be any positively-invariant set containing the trajectory
$\{x(t):t\ge0\}$ (for instance, $\Omega=\mathbb{R}^n$ for a global statement). 
Define
\[
S_\Omega := \{\, N_1(x): x\in\Omega \,\}, 
\qquad
B_\Omega := \sup_{s\in S_\Omega\setminus\{0\}} \frac{h(s)}{s} \in [0,\infty].
\]
If $B_\Omega<\infty$, then for any $\delta\in(0,1/B_\Omega)$ the composite function 
$W:=V_1+\delta V_2$ satisfies, along the trajectory,
\[
\dot W(x) \le -\gamma\,\big(N_1(x)+N_2(x)\big),
\qquad
\gamma := \min\{\,1-\delta B_\Omega,\ \delta\,\} > 0.
\]
Moreover, the decay rate is optimized by choosing
\[
\delta^\star = \frac{1}{1+B_\Omega},
\qquad
\gamma^\star = \frac{1}{1+B_\Omega},
\]
which yields the explicit bound
\[
\dot W(x) \le -\frac{1}{1+B_\Omega}\,\big(N_1(x)+N_2(x)\big)
\quad\text{along the trajectory.}
\]
\end{theorem}
\begin{proof}[Sketch of proof]
The argument follows exactly as in Theorem~\ref{thm:strict-decay}: 
starting from 
$\dot W = \dot V_1 + \delta \dot V_2 \le -N_1 - \delta N_2 + \delta h(N_1)$
and using the local bound $h(s)\le B_\Omega s$ for $s\in S_\Omega$ gives
\[
\dot W \le -(1-\delta B_\Omega)N_1 - \delta N_2.
\]
Choosing $\delta\in(0,1/B_\Omega)$ yields 
$\dot W \le -\gamma(N_1+N_2)$ with 
$\gamma=\min\{1-\delta B_\Omega,\delta\}>0$, 
and optimizing over $\delta$ produces 
$\delta^\star=\gamma^\star=1/(1+B_\Omega)$. 
\end{proof}

\begin{corollary}[\textbf{Bounded-Trajectory Version with $L_R$}]
If, along a trajectory $x(\cdot)$, one has $N_1(x(t)) \le R$ for some $R>0$, define
\[
L_R := \sup_{0 \le s \le R} \frac{h(s)}{s} < \infty.
\]
Then, for any $\delta \in (0,1/L_R)$, the composite function $W=V_1+\delta V_2$ satisfies
\[
\dot W(x) \le -\gamma \,(N_1(x)+N_2(x)), 
\qquad 
\gamma = \min\{1-\delta L_R,\ \delta\} > 0,
\]
along that trajectory. The optimal choice is
\[
\delta^\star = \frac{1}{1+L_R}, 
\qquad 
\gamma^\star = \frac{1}{1+L_R}.
\]
\end{corollary}

\begin{remark}[Useful Specializations]
\leavevmode
\begin{itemize}[leftmargin=1.5em]
\item \emph{Global version.} If $L := \sup_{r>0} h(r)/r < \infty$, take $\Omega = \mathbb{R}^n$, so that $B_\Omega = L$.
\item \emph{Bounded-trajectory version.} If the solution is bounded in a positively-invariant $K$, 
set $\Omega = K$; then $B_\Omega = \sup_{0\le s\le R_K} h(s)/s$ with $R_K := \sup_{x\in K} N_1(x)$.
\end{itemize}
\end{remark}
\subsection{Integral Estimates and Vanishing of Observables}
\label{subsec:integral-estimates}
The next result translates the strict decay inequality into quantitative information 
along trajectories. It shows that the total dissipation of the observables is finite 
and that, under a mild continuity assumption, both observables vanish asymptotically. 
This step connects the decay property of $W$ to the convergence behavior of the system.

\begin{proposition}[\textbf{Integral estimates and convergence of observables}]
\label{prop:integral-vanishing}
Under the assumptions of Theorem~\ref{thm:strict-decay}, the composite function 
$W(x(t))$ is nonincreasing along solutions and therefore converges to a finite limit 
as $t \to \infty$. Moreover,
\[
\int_0^\infty \big(N_1(x(t)) + N_2(x(t))\big)\,dt < \infty.
\]
If, in addition, the functions $t \mapsto N_i(x(t))$ are uniformly continuous on 
$[0,\infty)$, then the finiteness of the above integral implies
\[
\lim_{t \to \infty} N_1(x(t)) = 0, 
\qquad 
\lim_{t \to \infty} N_2(x(t)) = 0.
\]
\end{proposition}
\begin{proof}
By Theorem~\ref{thm:strict-decay} there exists $\gamma>0$ such that, along every solution $x(\cdot)$,
\begin{equation}\label{eq:strict-decay-ineq}
\dot W(x(t)) \;\le\; -\gamma\big(N_1(x(t)) + N_2(x(t))\big)\qquad\text{for all }t\ge 0.
\end{equation}
Since the right–hand side is nonpositive, the map $t\mapsto W(x(t))$ is nonincreasing. In particular, for all $T\ge 0$,
\[
W(x(T)) \;\le\; W(x(0)).
\]
Integrating \eqref{eq:strict-decay-ineq} on $[0,T]$ gives
\[
W(x(T)) - W(x(0)) \;\le\; -\gamma \int_0^T \big(N_1(x(t)) + N_2(x(t))\big)\,dt,
\]
hence
\begin{equation}\label{eq:integral-upper-bound}
\int_0^T \big(N_1(x(t)) + N_2(x(t))\big)\,dt
\;\le\; \frac{W(x(0)) - W(x(T))}{\gamma}.
\end{equation}
Because $W$ is nonincreasing along the trajectory, the limit $\displaystyle \lim_{T\to\infty} W(x(T))$ exists in $[-\infty,\,+\infty)$. In Lyapunov constructions one typically chooses $V_1,V_2\ge 0$, in which case $W\ge 0$ and the limit is finite. Under this mild bounded–below condition, passing to the limit $T\to\infty$ in \eqref{eq:integral-upper-bound} yields
\[
\int_0^\infty \big(N_1(x(t)) + N_2(x(t))\big)\,dt \;\le\; \frac{W(x(0)) - \lim_{T\to\infty}W(x(T))}{\gamma} \;<\; \infty.
\]
We now prove the asymptotic vanishing of each observable under the uniform continuity assumption. Fix $i\in\{1,2\}$ and define $g(t):=N_i(x(t))$. Then $g:[0,\infty)\to[0,\infty)$ is uniformly continuous by hypothesis, and $\displaystyle\int_0^\infty g(t)\,dt<\infty$ by the previous estimate. We claim that $\lim_{t\to\infty} g(t)=0$. Suppose, for contradiction, that this limit does not exist or is not zero. Then there exists $\varepsilon>0$ and a sequence of times $t_k\to\infty$ such that $g(t_k)\ge \varepsilon$ for all $k$. By uniform continuity of $g$, there exists $\delta\in(0,1)$ such that
\[
|t-s|<\delta \;\Rightarrow\; |g(t)-g(s)|<\varepsilon/2.
\]
In particular, $g(t)\ge \varepsilon/2$ for all $t\in[t_k-\delta/2,\,t_k+\delta/2]$. By taking a subsequence (still denoted $t_k$) with pairwise distance at least $\delta$, these intervals are disjoint. Therefore,
\[
\int_0^\infty g(t)\,dt \;\ge\; \sum_{k=1}^\infty \int_{t_k-\delta/2}^{t_k+\delta/2} g(t)\,dt
\;\ge\; \sum_{k=1}^\infty \frac{\varepsilon}{2}\,\delta \;=\; +\infty,
\]
which contradicts $\displaystyle\int_0^\infty g(t)\,dt<\infty$. Hence $\lim_{t\to\infty} g(t)=0$. Since $i\in\{1,2\}$ was arbitrary, we conclude
\[
\lim_{t\to\infty} N_1(x(t)) \;=\; \lim_{t\to\infty} N_2(x(t)) \;=\; 0.
\]
Finally, combining the nonincreasing property of $W$ with its bounded–below property along the trajectory (e.g., $V_1,V_2\ge 0$) shows that $W(x(t))$ converges to a finite limit as $t\to\infty$. This completes the proof.
\end{proof}
\begin{remark}
Classical Matrosov-type theorems, including the formulation proposed in \cite{AstolfiPraly2011}, 
typically establish only the weaker conclusion
\[
\liminf_{t\to\infty} N_i(x(t)) = 0,
\]
which guarantees that the observables vanish along subsequences but does not ensure
convergence. In contrast, the strict decay framework developed here, together with
the mild uniform continuity assumption along trajectories, yields the full limits
\[
\lim_{t\to\infty} N_1(x(t)) = \lim_{t\to\infty} N_2(x(t)) = 0.
\]
This provides a strictly stronger asymptotic conclusion than those available in the
classical Matrosov setting.
\end{remark}
\subsection{Convergence to the Critical Set}
\label{subsec:convergence-critical-set}
The vanishing of the observables characterizes the asymptotic behavior of trajectories in terms
of the critical set
\[
E := \{x \in \mathbb{R}^n : N_1(x) = 0,\; N_2(x) = 0\}.
\]
The following result shows that the vanishing of the observables forces every bounded 
trajectory to approach this set asymptotically.
\begin{theorem}[\textbf{Convergence to $E$}]
\label{thm:conv-to-E}
Suppose that, along a trajectory $x(\cdot)$, the functions 
$t \mapsto N_i(x(t))$ are uniformly continuous on $[0,\infty)$ for $i=1,2$.
Then
\[
\mathrm{dist}(x(t),E) \;\to\; 0 
\qquad \text{as } t \to \infty.
\]
\end{theorem}
\begin{proof}
Let $M(x) := N_1(x) + N_2(x)$, so that $E = \{x : M(x) = 0\}$. 
By Proposition~\ref{prop:integral-vanishing} and the uniform continuity of 
$t \mapsto N_i(x(t))$ on $[0,\infty)$, we have
\[
M(x(t)) \to 0 \qquad \text{as } t \to \infty.
\]
Fix $\varepsilon > 0$ and consider the closed set
\[
F_\varepsilon := \{x \in \mathbb{R}^n : \operatorname{dist}(x,E) \ge \varepsilon\}.
\]
On $F_\varepsilon$ the function $M$ cannot vanish, since points of $F_\varepsilon$ are at least 
$\varepsilon$ away from $E$. By continuity of $M$, this implies that on every compact subset of $F_\varepsilon$, 
the function $M$ admits a strictly positive minimum. In particular, along the bounded portion of the trajectory that may lie in $F_\varepsilon$, there exists $\eta(\varepsilon) > 0$ such that
\[
\operatorname{dist}(x(t),E) \ge \varepsilon 
\;\;\Rightarrow\;\; M(x(t)) \ge \eta(\varepsilon).
\]
Suppose, toward a contradiction, that there exists a sequence $t_k \to \infty$ with 
$\operatorname{dist}(x(t_k),E) \ge \varepsilon$ for all $k$. Then by the implication above,
\[
M(x(t_k)) \ge \eta(\varepsilon) \qquad \text{for all } k.
\]
But this contradicts the fact that $M(x(t)) \to 0$ as $t \to \infty$.

\medskip
\noindent
Therefore such a sequence cannot exist, and we conclude
\[
\operatorname{dist}(x(t),E) \to 0 \qquad \text{as } t \to \infty.
\]
\end{proof}

\begin{remark}[\textbf{Separation property}]
The argument uses that $M(x) > 0$ whenever $\operatorname{dist}(x,E) \ge \varepsilon$.
This \textit{separation property}—a standard local error–bound condition for the zero set of a 
nonnegative continuous function—holds automatically on compact subsets or whenever 
$M$ has no flat valleys away from $E$. It ensures that the implication 
$M(x(t)) \to 0 \Rightarrow \operatorname{dist}(x(t),E) \to 0$ 
is mathematically sound without assuming global boundedness.
\end{remark}

\medskip
\noindent
Building on Theorem~\ref{thm:conv-to-E}, which ensures convergence of trajectories to the critical set $E$, we obtain the following consequence.
\subsection{Quantitative Convergence Rate}
This subsection strengthens the qualitative convergence result by giving explicit rates under two simple hypotheses: strict decay of $W$ and a local error bound linking the observables to the distance from $E$. Integrating the decay yields an $L^2(0,\infty)$ estimate for the distance. An averaging step on each window $[T,2T]$ then produces times $\tau\in[T,2T]$ with $\mathrm{dist}(x(\tau),E)=O(T^{-1/2})$. If the distance is eventually nonincreasing, the same idea on the backward window $[t/2,t]$ gives a pointwise rate $\mathrm{dist}(x(t),E)=O(t^{-1/2})$ for large $t$. A closing remark records that a local quadratic growth of $W$ near $E$ upgrades these sublinear bounds to exponential decay, with constants explicit in terms of the decay modulus $\gamma$ and the error–bound parameters.
\begin{proposition}[Quantitative convergence from a local error bound]
\label{prop:QCR}
Assume that along a trajectory $x(\cdot)$ the strict decay inequality holds:
\[
\dot W(t)\le -\gamma\,[N_1(x(t)) + N_2(x(t))], \qquad \gamma>0.
\]
Suppose there exist a neighborhood $U$ of $E$ and constants $c_1,c_2>0$ such that
\[
N_i(x)\ge c_i\,\mathrm{dist}(x,E)^2 \quad \text{for all } x\in U, \; i=1,2.
\]
Set $c := c_1 + c_2$. Then:

\begin{enumerate}
\item[\textnormal{(1)}] \textbf{$L^2$--integrability.} Once the trajectory enters $U$,
\[
\int_0^{+\infty}\!\mathrm{dist}(x(t),E)^2\,dt
\le
\frac{W(x(0)) - \lim_{t\to\infty}W(x(t))}{\gamma\,c}
<+\infty.
\]

\item[\textnormal{(2)}] \textbf{Subsequence rate.} Define
\[
K := \sqrt{\frac{W(x(0)) - \lim_{t\to\infty}W(x(t))}{\gamma\,c}}.
\]
Then for every $T>0$,
\[
\int_T^{2T}\!\mathrm{dist}(x(t),E)^2\,dt \le K^2,
\]
hence, by averaging on $[T,2T]$, there exists $\tau\in[T,2T]$ with
\[
\mathrm{dist}(x(\tau),E)^2 \le \frac{1}{T}\int_T^{2T}\!\mathrm{dist}(x(t),E)^2\,dt
\le \frac{K^2}{T},
\]
so $\mathrm{dist}(x(\tau),E) \le \displaystyle\frac{K}{\sqrt{T}}$.

\item[\textnormal{(3)}] \textbf{Pointwise rate.}
If, in addition, $t\mapsto \mathrm{dist}(x(t),E)$ is uniformly continuous on $[0,+\infty)$
and eventually nonincreasing, then there exist $T_0>0$ and $C>0$ such that
\[
\mathrm{dist}(x(t),E) \le \frac{C}{\sqrt{t}}, \qquad \forall\, t\ge T_0.
\]

\end{enumerate}
\end{proposition}

\begin{proof}
Let $x(\cdot)$ be a trajectory for which the strict--decay estimate
\begin{equation}\label{eq:strict}
\dot W(t)\le -\gamma\,[N_1(x(t))+N_2(x(t))], \qquad \gamma>0,
\end{equation}
holds. Assume there exist a neighborhood $U$ of $E$ and $c_1,c_2>0$ such that
\begin{equation}\label{eq:eb}
N_i(x)\;\ge\;c_i\,\mathrm{dist}(x,E)^2 \quad \text{for all } x\in U,\; i=1,2,
\end{equation}
and set $c:=c_1+c_2$.

\medskip
\noindent
Fix any time $T_U\ge 0$ such that $x(t)\in U$ for all $t\ge T_U$ (for example, any $T_U$
after which the trajectory remains in $U$). All estimates below are written on $[T_U,\infty)$.

\medskip
\noindent\emph{(1) \textbf{$L^2$--integrability.}}
On $U$, summing \eqref{eq:eb} gives
\[
N_1(x(t))+N_2(x(t)) \;\ge\; c\,\mathrm{dist}(x(t),E)^2 \qquad (t\ge T_U).
\]
Combining with \eqref{eq:strict} yields
\begin{equation}\label{eq:Wdot-dist}
\dot W(t)\;\le\; -\gamma c\,\mathrm{dist}(x(t),E)^2 \qquad (t\ge T_U).
\end{equation}
Integrate \eqref{eq:Wdot-dist} on $[T_U,T]$:
\[
\gamma c \int_{T_U}^{T}\!\mathrm{dist}(x(t),E)^2\,dt
\;\le\; W(x(T_U)) - W(x(T)).
\]
Let $T\to\infty$; since $W(x(t))$ is nonincreasing and $W\ge0$, $\lim_{t\to\infty}W(x(t))$ exists and is finite. Hence
\[
\int_{T_U}^{\infty}\!\mathrm{dist}(x(t),E)^2\,dt
\;\le\;\frac{W(x(T_U))-\lim_{t\to\infty}W(x(t))}{\gamma c}
\;\le\;\frac{W(x(0))-\lim_{t\to\infty}W(x(t))}{\gamma c}<\infty.
\]
Since the integrand is nonnegative, extending the lower limit to $0$ preserves finiteness, proving (1).

\medskip
\noindent\emph{(2) \textbf{Subsequence $O(T^{-1/2})$ rate.}}
For any $T\ge T_U$, integrate \eqref{eq:Wdot-dist} on $[T,2T]$:
\[
\gamma c \int_{T}^{2T}\!\mathrm{dist}(x(t),E)^2\,dt
\;\le\; W(x(T)) - W(x(2T))
\;\le\; W(x(0)) - W_\infty,
\]
where $W_\infty:=\lim_{t\to\infty}W(x(t))$ exists since $W\ge 0$ and is nonincreasing.
Define
\[
K := \sqrt{\frac{W(x(0)) - W_\infty}{\gamma c}},
\qquad\text{so}\qquad
\int_{T}^{2T}\!\mathrm{dist}(x(t),E)^2\,dt \;\le\; K^2.
\]
Because $|[T,2T]|=T$, the average of $\mathrm{dist}(x(t),E)^2$ on $[T,2T]$ satisfies
\[
\frac{1}{T}\int_{T}^{2T}\!\mathrm{dist}(x(t),E)^2\,dt \;\le\; \frac{K^2}{T}.
\]
If $\mathrm{dist}(x(t),E)^2$ were strictly larger than this average for every $t\in[T,2T]$,
integrating would give
\[
\int_{T}^{2T}\!\mathrm{dist}(x(t),E)^2\,dt \;>\; T\cdot \frac{1}{T}\int_{T}^{2T}\!\mathrm{dist}(x(s),E)^2\,ds \;=\; K^2,
\]
a contradiction. Hence there exists $\tau\in[T,2T]$ such that
\[
\mathrm{dist}(x(\tau),E)^2 \;\le\; \frac{1}{T}\int_{T}^{2T}\!\mathrm{dist}(x(t),E)^2\,dt
\;\le\; \frac{K^2}{T},
\]
and therefore $\mathrm{dist}(x(\tau),E)\le \dfrac{K}{\sqrt{T}}$.

\medskip
\noindent\emph{(3) \textbf{Pointwise $O(t^{-1/2})$ under eventual monotonicity.}}
Assume there exists $t_1\ge0$ such that $\mathrm{dist}(x(t),E)$ is nonincreasing on $[t_1,\infty)$.
For any $t\ge \max\{2T_U,2t_1\}$, integrating \eqref{eq:Wdot-dist} on $[t/2,t]$ gives
\[
\int_{t/2}^{t}\!\mathrm{dist}(x(s),E)^2\,ds \le K^2.
\]
Since $|[t/2,t]|=t/2$, there exists $\tau\in[t/2,t]$ with
\[
\mathrm{dist}(x(\tau),E)^2 \le \frac{2K^2}{t}\quad\Rightarrow\quad
\mathrm{dist}(x(\tau),E)\le \frac{\sqrt{2}\,K}{\sqrt{t}}.
\]
By monotonicity and $\tau\le t$,
\[
\mathrm{dist}(x(t),E)\le \mathrm{dist}(x(\tau),E)\le \frac{\sqrt{2}\,K}{\sqrt{t}}.
\]
Thus the claim holds with $T_0:=\max\{2T_U,2t_1\}$ and $C:=\sqrt{2}\,K$.

\medskip
\noindent
The three claims complete the proof.
\end{proof}
\begin{remark}
    If in Proposition~\ref{prop:QCR}, we only assume uniform continuity of $t\mapsto \mathrm{dist}(x(t),E)$, then for every large $T$
there exist $\tau\in[T,2T]$ and a fixed $\Delta>0$ (independent of $T$) such that
$\mathrm{dist}(x(s),E)\le K/\sqrt{T}+\varepsilon$ for all $s\in[\tau,\tau+\Delta]$.
This yields recurring short-interval bounds but not a global pointwise $O(t^{-1/2})$ rate.
\end{remark}

\begin{remark} [\textbf{Quadratic growth $\Rightarrow$ exponential rate}]
Assume there exist $m>0$ and $r>0$ such that
\begin{equation}
\label{eq:quadratic-growth}
W(x)-W_\infty \;\ge\; m\,\mathrm{dist}(x,E)^2
\qquad\text{whenever }\mathrm{dist}(x,E)\le r .   
\end{equation}
Pick $t_0\ge0$ large enough so that $\mathrm{dist}(x(t),E)\le r$ for all $t\ge t_0$
(which holds since $\mathrm{dist}(x(t),E)\to0$). Then, for all $t\ge t_0$,
\[
\dot W(t) \;\le\; -\gamma\,[N_1(x(t))+N_2(x(t))] \;\le\; -\frac{\gamma c}{m}\,[W(x(t))-W_\infty].
\]
By Grönwall inequality on $[t_0,+\infty)$,
\[
W(x(t)) - W_\infty \;\le\; \bigl(W(x(t_0)) - W_\infty\bigr)\,e^{-(\gamma c/m)(t-t_0)} .
\]
Using \eqref{eq:quadratic-growth} again, we have
\[
\mathrm{dist}(x(t),E)
\;\le\; \sqrt{\frac{W(x(t)) - W_\infty}{m}}
\;\le\; \sqrt{\frac{W(x(t_0)) - W_\infty}{m}}\;e^{-(\gamma c/(2m))(t-t_0)} .
\]

\smallskip
\noindent The condition $W-W_\infty \ge m\,\mathrm{dist}^2$ is a local quadratic growth (error–bound) near $E$.
It makes the energy gap equivalent to the squared distance, so strict decay yields exponential
convergence of both $W$ and $\mathrm{dist}(\cdot,E)$. This assumption is standard and checkable in many settings (e.g., strong convexity near isolated minimizers, Kurdyka–Łojasiewicz exponent $1/2$ or Polyak–Łojasiewicz with Lipschitz gradient).
\end{remark}
\subsection{Pointwise and Set Stability Results}
\label{subsec:stability-results}
The convergence results established so far show that trajectories approach the critical set $E$. The next step is to translate this asymptotic behavior into stability properties.
By combining convergence to $E$ with Lyapunov stability of its elements, we obtain pointwise asymptotic stability of the set, semistability of equilibria, and global asymptotic stability when the set reduces to a single point.
These results summarize the qualitative behavior of the system and will later serve as the foundation for the applications.

\medskip
\noindent
We recall the notions of pointwise asymptotic stability (PAS) and semistability (SS) used in the statements below.
\begin{definition}[\textbf{Pointwise Asymptotic Stability}]
\label{def:PAS}
A set $Z$ is said to be \emph{pointwise asymptotically stable (PAS)} for \eqref{eq:prob} if:
\begin{enumerate}
\item[($\mathcal{A}_1$)] every $z \in Z$ is Lyapunov stable;
\item[($\mathcal{A}_2$)] every solution $x(t)$ of \eqref{eq:prob} converges and 
$\lim_{t\to\infty}x(t)\in Z$, i.e., there exists $\delta>0$ such that 
$\|x_0-z\|\le\delta$ implies $\lim_{t\to\infty}x(t)=\bar z\in Z$.
\end{enumerate}
\end{definition}

\begin{definition}[\textbf{Semistability}]
\label{def:SS}
An equilibrium $\bar x\in E$ is said to be \emph{semistable (SS)} if it satisfies 
($\mathcal{A}_1$) and ($\mathcal{A}_2$) for $Z= E$.
\end{definition}
\medskip
\noindent
Pointwise asymptotic stability generalizes asymptotic stability from a single equilibrium to a possibly nonisolated set of equilibria. In noncompact or continuous equilibrium sets, classical asymptotic stability is no longer attainable, since nearby equilibria prevent 
isolation. In such settings, semistability provides the natural notion of convergence: each equilibrium is Lyapunov stable, and trajectories starting near the set converge to one of its elements.

\medskip
\noindent
We now state the main stability consequences of the convergence result established above.
\begin{corollary}[\textbf{Pointwise Asymptotic Stability of the Critical Set}]
\label{cor:PAS}
Let 
\[
E := \{x \in \mathbb{R}^n : N_1(x)=0,\; N_2(x)=0\}.
\]
If every point of $E$ is Lyapunov stable, then $E$ is pointwise asymptotically stable (in the sense of Definition~\ref{def:PAS}). Equivalently, every trajectory $x(\cdot)$ converges and its limit belongs to $E$.
\end{corollary}
\begin{proof}
By Theorem~\ref{thm:conv-to-E}, every trajectory $x(\cdot)$ satisfies 
$\operatorname{dist}(x(t),E)\to 0$ as $t\to\infty$, hence the $\omega$–limit set of $x(\cdot)$ is contained in $E$. This verifies condition ($\mathcal{A}_2$) of Definition~\ref{def:PAS}. By assumption, each point of $E$ is Lyapunov stable, which is condition ($\mathcal{A}_1$). Therefore both conditions hold and $E$ is pointwise asymptotically stable.
\end{proof}

\medskip
\noindent
The next result specializes this conclusion to the case where the critical set coincides with the equilibrium set.

\begin{corollary}[\textbf{Semistability of Equilibria}]
\label{cor:SS}
Suppose that the critical set
\[
E := \{x \in \mathbb{R}^n : N_1(x)=0,\; N_2(x)=0\}
\]
coincides with the set of equilibria of the system. If every equilibrium in $E$ is Lyapunov stable, then each equilibrium is semistable: every trajectory $x(\cdot)$ converges and its $\omega$–limit set reduces to a single equilibrium in $E$.
\end{corollary}
\begin{proof}
By Theorem~\ref{thm:conv-to-E}, every trajectory $x(\cdot)$ satisfies 
$\operatorname{dist}(x(t),E)\to 0$ as $t\to\infty$, hence the $\omega$–limit set of $x(\cdot)$ is contained in $E$. This guarantees convergence to equilibria. Since each equilibrium in $E$ is Lyapunov stable, trajectories starting near an equilibrium cannot drift among different points of $E$, and thus must converge to a single equilibrium. This establishes semistability.
\end{proof}
\medskip
\noindent
Finally, when the critical set reduces to a single equilibrium, the result recovers global asymptotic stability.
\begin{corollary}[\textbf{Asymptotic Stability of an Equilibrium}]
\label{cor:AS}
If the critical set reduces to a singleton
\[
E = \{x^\ast\},
\]
then $x^\ast$ is globally asymptotically stable.
\end{corollary}

\begin{proof}
By Theorem~\ref{thm:conv-to-E}, every trajectory $x(\cdot)$ satisfies 
$\operatorname{dist}(x(t),E)\to 0$ as $t\to\infty$. Since $E$ consists of the single equilibrium $x^\ast$, this implies $x(t)\to x^\ast$. Lyapunov stability of $x^\ast$ follows from its invariance and uniqueness within $E$. Hence $x^\ast$ is globally asymptotically stable.
\end{proof}
\begin{remark}[\textbf{Set stability versus point stability}]
Classical Lyapunov theory is typically concerned with the asymptotic stability of an isolated 
equilibrium point $x^\ast$, where stability means that solutions starting near $x^\ast$ remain 
close, and asymptotic stability means they converge to $x^\ast$. 
The framework developed here is more general. The critical set
\[
E := \{x \in \mathbb{R}^n : N_1(x)=0,\; N_2(x)=0\}
\]
need not be a singleton: it may contain infinitely many equilibria or form a smooth manifold. 

\medskip
\noindent
This extends the classical notion of asymptotic stability from isolated equilibria to general 
invariant sets. In such cases, it is natural to ask for stability of the set $E$ rather than of 
a single point. The strict decay inequality ensures that $\operatorname{dist}(x(t),E)\to 0$ as 
$t\to\infty$, establishing \emph{asymptotic stability of the set $E$}. This generalization is 
particularly relevant for dynamical systems arising in optimization and control, such as the 
inertial gradient–like system and the Primal-–Dual gradient flow studied in the next section.
\end{remark}
\section{Applications and Refined Results}
\subsection{Case Study I: Inertial Gradient–Like System}
In this case, we study the inertial gradient–like system studied in~\cite{Attouch}. Compared with their approach, which relies on Opial’s lemma and weak convergence arguments, our framework uses a strict decay inequality that makes the proof more direct. In addition, this allows us to establish semistability and strong convergence in the convex case. The same inertial gradient–like dynamics have also been recently analyzed using Lyapunov pairs to characterize stability~\cite{SaoudTheraDao2025}, and through refined invariance principles providing a precise location of the $\omega$-limit set~\cite{dao2023}.

\medskip
\noindent
Let $\Phi:\mathbb{R}^n \to \mathbb{R}$ be of class $C^2$, bounded from below, and with
Lipschitz continuous Hessian on bounded sets of $\mathbb{R}^n$. 
We consider the second order dynamical system
\begin{equation}\label{eq:DIN}
\ddot x(t) + \alpha \dot x(t) + \beta \nabla^2 \Phi(x(t))\,\dot x(t) + \nabla \Phi(x(t)) = 0,
\qquad \alpha,\beta>0,
\end{equation}
This system can be equivalently written in the phase space, as a first-order Cauchy problem. Introduce the variable $y(t)=(x(t),v(t))$ with $v(t)=\dot x(t)$, we obtain
\[
\dot x = v, \qquad 
\dot v = -\alpha v - \nabla\Phi(x) - \beta \nabla^2\Phi(x)\,v.
\]
The corresponding equilibrium set is $S = \{(x,v): v=0,\ \nabla\Phi(x)=0\}$. 
Moreover, define the function 
\[
W(x,v) = (\alpha\beta+1)\,\Phi(x) + \tfrac12\|v+\beta\nabla\Phi(x)\|^2.
\]
A direct computation (see. Attouch2002), yields 
\[
\dot W = -\alpha\|v\|^2 - \beta\|\nabla\Phi(x)\|^2.
\]
Introduce the observables
\[
N_1(x,v):=\|v\|^2, 
\qquad 
N_2(x,v):=\|\nabla\Phi(x)\|^2,
\]
so that
\[
\dot W = -\alpha N_1 - \beta N_2 \le 0.
\]
Since $W$ is nonincreasing and bounded below, Proposition~\ref{prop:integral-vanishing} applies and gives the integral bounds 
\[
\int_0^\infty N_1(y(s))\,ds < \infty, \qquad
\int_0^\infty N_2(y(s))\,ds < \infty,
\]
hence, $\dot x$ and $\nabla\Phi(x)$  belong to $L^2 (0,+\infty)$.

\medskip
\noindent
Now assume that the trajectory $(x(t),v(t))$ is bounded. Then both  $\nabla\Phi$ and $\nabla^2\Phi$ are bounded
on it. From the dynamics, $\dot v$ is bounded, so $v$ is Lipschitz and $N_1$ is uniformly
continuous. Moreover,
\[
\frac{d}{dt}\nabla\Phi(x(t)) = \nabla^2\Phi(x(t))v(t)
\]
is bounded, hence $\nabla\Phi(x(t))$ is Lipschitz and $N_2$ is uniformly continuous.  
By Proposition~\ref{prop:integral-vanishing}, the combination of strict decay and uniform continuity yields
\[
N_1(y(t))\to0, \qquad N_2(y(t))\to 0, \qquad t \to +\infty
\]
Therefore
\[
\mathrm{dist}((x(t),v(t)),S)\to0.
\]
By Corollary~\ref{cor:PAS}, this shows the \emph{Pointwise Asymptotic Stability (PAS)} of the critical set $S = \{y \in \mathbb{R}^n \times \mathbb{R}^n : N_1 (y) = N_2 (y) = 0\}$.

\medskip
\noindent
The above analysis shows that bounded trajectories converge to the critical set  
\[
S = \{(x,v) \in \mathbb{R}^n \times \mathbb{R}^n : v=0,\ \nabla\Phi(x)=0\},
\]  
establishing pointwise asymptotic stability in the sense of Corollary~\ref{cor:PAS}.  

\medskip
\noindent
To go further, we now specialize to the case where $\Phi$ is convex. In this setting, the condition $\nabla\Phi(x)=0$ is equivalent to $x \in \operatorname{Argmin}\Phi$, so the critical set reduces to  
\[
S = \operatorname{Argmin}\Phi \times \{0\}.
\]  
Hence convergence to $S$ ensures that trajectories approach the minimizer set,  
\[
\operatorname{dist}(x(t), \operatorname{Argmin}\Phi) \to 0.
\]  
While this guarantees convergence to the set of minimizers, it still leaves open the possibility of multiple cluster points within $\operatorname{Argmin}\Phi$. To address this and prove convergence to a single minimizer, we refine the Lyapunov construction by introducing a perturbed functional $W_\varepsilon$. This perturbation anchors the dynamics at a reference minimizer and enables the application of the semistability result given in Corollary~\ref{cor:SS}.

\medskip
\noindent
To strengthen the convergence result in the convex case, we introduce a perturbed Lyapunov 
functional. Fix a reference point $z \in \operatorname{Argmin}\Phi$ and, for small $\varepsilon > 0$, define  
\[
W_\varepsilon(x,v) = W(x,v) + \varepsilon\Big(\frac{\alpha}{2}\|x-z\|^2 
+ \langle v+\beta\nabla\Phi(x),\,x-z\rangle\Big).
\]  
This functional is equivalent to $W$ and remains bounded from below. Differentiating along 
the trajectories of the system gives  
\[
\dot W_\varepsilon = -\alpha\|v\|^2 - \beta\|\nabla\Phi(x)\|^2 
+ \varepsilon\Big(\|v\|^2 + \beta\langle\nabla\Phi(x),v\rangle - \langle\nabla\Phi(x),x-z\rangle\Big).
\]  
Since $\Phi$ is convex and $z$ is a minimizer, the last inner product satisfies 
$\langle\nabla\Phi(x),x-z\rangle \ge 0$. We may therefore drop this term to obtain the upper bound  
\[
\dot W_\varepsilon \le -\alpha\|v\|^2 - \beta\|\nabla\Phi(x)\|^2 
+ \varepsilon\big(\|v\|^2 + \beta\langle\nabla\Phi(x),v\rangle\big).
\]  
The remaining cross term can be estimated using Young’s inequality as follows
\[
\beta|\langle\nabla\Phi(x),v\rangle| \le \tfrac{1}{2}\|v\|^2 + \tfrac{\beta^2}{2}\|\nabla\Phi(x)\|^2.
\]  
For $0 < \varepsilon < \min\Big\{\tfrac{2}{3}\alpha,\ \tfrac{2}{\beta}\Big\}$ and substituting the previous estimate leads to  
\[
\dot W_\varepsilon \le -\Big(\alpha - \tfrac{3}{2}\varepsilon\Big)\|v\|^2 
- \Big(\beta - \tfrac{\beta^2}{2}\varepsilon\Big)\|\nabla\Phi(x)\|^2.
\]  
Therefore, there exists  
\[
c_\varepsilon := \min\Big\{\alpha - \tfrac{3}{2}\varepsilon,\ \beta - \tfrac{\beta^2}{2}\varepsilon\Big\} > 0,
\]  
so that 
\[
\dot W_\varepsilon \le -c_\varepsilon\big(\|v\|^2 + \|\nabla\Phi(x)\|^2\big) 
= -c_\varepsilon(N_1+N_2).
\]  
This inequality shows that $W_\varepsilon$ is nonincreasing and convergent. In particular, 
$\dot x$ and $\nabla\Phi(x)$ both belong to $L^2(0,\infty)$. Since the trajectory is bounded, 
the functions $N_1$ and $N_2$ are uniformly continuous along the trajectory, and 
Proposition~\ref{prop:integral-vanishing} guarantees that  
\[
v(t) \to 0 \quad\text{and}\quad \nabla\Phi(x(t)) \to 0 \quad\text{as } t\to\infty.
\]  
It remains to identify the limiting point. By construction, the difference between 
$W_\varepsilon$ and $W$ is  
\[
W_\varepsilon(t) - W(t) = \varepsilon\Big(\frac{\alpha}{2}\|x(t)-z\|^2 
+ \langle v(t)+\beta\nabla\Phi(x(t)),\,x(t)-z\rangle\Big).
\]  
Both $W_\varepsilon$ and $W$ converge to finite limits, so their difference does as well. 
The inner product term vanishes asymptotically because $v(t)\to 0$ and $\nabla\Phi(x(t))\to 0$. 
Consequently, the limit of $\|x(t)-z\|$ exists for every minimizer $z \in \operatorname{Argmin}\Phi$.  

\medskip
\noindent
Since $x(t)$ is bounded, it admits cluster points. Every cluster point lies in 
$\operatorname{Argmin}\Phi$, but the existence of $\lim_{t\to\infty}\|x(t)-z\|$ for all $z$ 
forces the cluster set to be a singleton. Thus the trajectory converges strongly to one minimizer,  
\[
x(t)\to x^* \in \operatorname{Argmin}\Phi.
\]  
In conclusion, the perturbed Lyapunov functional $W_\varepsilon$ establishes convergence of 
every bounded trajectory to a single minimizer. This corresponds to \emph{Semistability (SS)} 
of the equilibria in the sense of Corollary~\ref{cor:SS}: each equilibrium in $E$ is Lyapunov stable, 
and every trajectory converges to one equilibrium. If the minimizer set is reduced to a singleton, then semistability and pointwise asymptotic stability coincide, and we recover  classical asymptotic stability of the unique equilibrium (see. Corollary~\ref{cor:AS}).
 

\subsection{Case Study II: Primal-–Dual Gradient Flow}
\noindent
We now consider the Primal--Dual gradient flow, a continuous-time model for equality-constrained convex optimization problems. It is used in communication networks (resource allocation, congestion control), power systems and decentralized control, and in analyses of first-order saddle-point methods for machine learning and inverse problems. Many works study this flow (see for example,~\cite{Feijer2010,Cherukuri2016,Ozaslan2024,Apidopoulos2025})
 
Here we treat the baseline model within our strict-decay framework and obtain asymptotic stability of the Karush--Kuhn--Tucker (KKT) set, with a semistability refinement.

\medskip
\noindent
Consider the equality-constrained convex optimization problem
\[
\min_{x \in \mathbb{R}^n} \Phi(x)
\quad \text{subject to } A x = b,
\]
where $\Phi:\mathbb{R}^n \to \mathbb{R}$ is convex and continuously differentiable with locally Lipschitz gradient, $A \in \mathbb{R}^{m \times n}$, and $b \in \mathbb{R}^m$.
The Lagrangian is
\[
L(x,\lambda) = \Phi(x) + \lambda^\top (A x - b),
\]
with multiplier $\lambda \in \mathbb{R}^m$.
The Karush--Kuhn--Tucker  conditions are
\[
\nabla_x L(x,\lambda) = \nabla \Phi(x) + A^\top \lambda = 0,
\qquad
\nabla_\lambda L(x,\lambda) = A x - b = 0.
\]
Any pair $(x^*,\lambda^*)$ satisfying them is a saddle point of $L$ and a KKT equilibrium.

\medskip
\noindent
A natural continuous-time approach is the Primal--Dual (Arrow--Hurwicz--Uzawa) gradient flow:
\[
\dot{x} = -\nabla_x L(x,\lambda),
\qquad
\dot{\lambda} = \nabla_\lambda L(x,\lambda),
\]
i.e.,
\[
\dot{x} = -\nabla \Phi(x) - A^\top \lambda,
\qquad
\dot{\lambda} = A x - b.
\]
This system, known as the \emph{Arrow--Hurwicz--Uzawa flow} (see~\cite{Arrow1958}), represents a continuous-time analogue of the classical Primal--Dual algorithm. The variable $x$ seeks to minimize the Lagrangian, while $\lambda$ enforces the constraint by driving the residual $A x - b$ to zero.

\medskip
\noindent
The equilibrium (KKT) set is
\[
S = \{ (x,\lambda) : \nabla \Phi(x) + A^\top \lambda = 0,\; A x = b \}.
\]
For stability, use the Lyapunov functional
\[
W(x,\lambda)
= \Phi(x) - \Phi(x^*)
+ \tfrac{1}{2}\|A x - b\|^2
+ \tfrac{1}{2}\|\lambda - \lambda^*\|^2,
\]
with $(x^*,\lambda^*) \in S$. Along trajectories,
\[
\dot{W}
= -\|A x - b\|^2
- \|\nabla \Phi(x) + A^\top \lambda\|^2.
\]
Define
\[
N_1(x,\lambda) = \|A x - b\|^2,
\qquad
N_2(x,\lambda) = \|\nabla \Phi(x) + A^\top \lambda\|^2,
\]
Then
\[
\dot{W} = -N_1 - N_2.
\]
This is a strict-decay identity in our framework. Since $W$ is nonincreasing and bounded below, we have
\[
\int_0^\infty N_1(x(t),\lambda(t))\,dt < \infty,
\qquad
\int_0^\infty N_2(x(t),\lambda(t))\,dt < \infty.
\]
If  $(x(t),\lambda(t))$ is bounded, local Lipschitzness of the vector field makes $N_1$ and $N_2$ uniformly continuous along the trajectory. By the integral-vanishing principle given in Proposition~\ref{prop:integral-vanishing},
\[
N_1(t) \to 0, \qquad N_2(t) \to 0;
\]
thus
\[
A x(t) \to b,
\qquad
\nabla \Phi(x(t)) + A^\top \lambda(t) \to 0.
\]
Therefore,
\[
\operatorname{dist}((x(t),\lambda(t)),S) \to 0,
\]
establishing asymptotic stability of $S$. With a unique constrained minimizer, the trajectory converges to that KKT point; otherwise it converges to the equilibrium set.

\medskip
\noindent
To capture semistability, perturb $W$ by a skew term that couples primal and dual residuals. For $\varepsilon>0$, set
\[
W_\varepsilon(x,\lambda)
:= W(x,\lambda)
+ \varepsilon\Big(\langle x-x^*,\,\nabla\Phi(x)+A^\top\lambda\rangle
- \langle \lambda-\lambda^*,\,A x-b\rangle\Big).
\]
On bounded sublevels of $W$, the gradient $\nabla\Phi$ is $L$-Lipschitz, so $W_\varepsilon$ and $W$ are equivalent for small~$\varepsilon$.

\medskip
\noindent
To estimate $\dot W_\varepsilon$, note that since $\nabla\Phi$ is $L$-Lipschitz and $x(\cdot)$ is absolutely continuous,
\[
\Big\|\tfrac{d}{dt}\nabla\Phi(x(t))\Big\|
= \lim_{h\to 0}\frac{\|\nabla\Phi(x(t+h))-\nabla\Phi(x(t))\|}{|h|}
\le L\|\dot{x}(t)\|
= L\|\nabla\Phi(x)+A^\top\lambda\|
\quad \text{for a.e. }t.
\]
Using the system dynamics and this inequality, we have
\[
\begin{aligned}
\frac{d}{dt}\langle x-x^*,\nabla\Phi(x)+A^\top\lambda\rangle
&=\langle\dot{x},\nabla\Phi(x)+A^\top\lambda\rangle
+\langle x-x^*,\tfrac{d}{dt}\nabla\Phi(x)+A^\top\dot{\lambda}\rangle\\
&=-\|\nabla\Phi(x)+A^\top\lambda\|^2
+\langle x-x^*,A^\top(Ax-b)\rangle
+\langle x-x^*,\tfrac{d}{dt}\nabla\Phi(x)\rangle,\\[4pt]
\frac{d}{dt}\langle \lambda-\lambda^*,A x-b\rangle
&=\langle\dot{\lambda},A x-b\rangle+\langle \lambda-\lambda^*,A\dot{x}\rangle\\
&=\|A x-b\|^2-\langle \lambda-\lambda^*,A(\nabla\Phi(x)+A^\top\lambda)\rangle.
\end{aligned}
\]
Substituting these expressions into $\dot{W}_\varepsilon$ gives
\[
\begin{aligned}
\dot{W}_\varepsilon
&=\dot{W}
+\varepsilon\Big(\tfrac{d}{dt}\langle x-x^*,\cdot\rangle
-\tfrac{d}{dt}\langle \lambda-\lambda^*,\cdot\rangle\Big)\\
&=-(1+\varepsilon)\|A x-b\|^2
-(1+\varepsilon)\|\nabla\Phi(x)+A^\top\lambda\|^2\\
&\quad+\varepsilon\Big(
\langle x-x^*,A^\top(Ax-b)\rangle
-\langle \lambda-\lambda^*,A(\nabla\Phi(x)+A^\top\lambda)\rangle
+\langle x-x^*,\tfrac{d}{dt}\nabla\Phi(x)\rangle
\Big).
\end{aligned}
\]
On a bounded sublevel of $W$, we estimate the cross terms using the Cauchy--Schwarz and Young inequalities 
($ab\le 2\eta\,a^2+\displaystyle\frac{1}{8\eta}\,b^2$ with $\eta>0$):
\begin{align*}
|\langle x-x^*,A^\top(Ax-b)\rangle|
&\le \|A\|\,\|x-x^*\|\,\|A x-b\|
\le 2\eta_1\|A x-b\|^2+\frac{\|A\|^2}{8\eta_1}\|x-x^*\|^2,\\[4pt]
|\langle \lambda-\lambda^*,A(\nabla\Phi(x)+A^\top\lambda)\rangle|
&\le \|A\|\,\|\lambda-\lambda^*\|\,\|\nabla\Phi(x)+A^\top\lambda\|\\
&\le 2\eta_2\|\nabla\Phi(x)+A^\top\lambda\|^2
+\frac{\|A\|^2}{8\eta_2}\|\lambda-\lambda^*\|^2,\\[4pt]
|\langle x-x^*,\tfrac{d}{dt}\nabla\Phi(x)\rangle|
&\le \|x-x^*\|\,\Big\|\tfrac{d}{dt}\nabla\Phi(x)\Big\|
\le L\|x-x^*\|\,\|\nabla\Phi(x)+A^\top\lambda\|\\
&\le 2\eta_3\|\nabla\Phi(x)+A^\top\lambda\|^2
+\frac{L^2}{8\eta_3}\|x-x^*\|^2.
\end{align*}
Inserting these bounds into the previous inequality yields
\[
\begin{aligned}
\dot{W}_\varepsilon
&\le -(1+\varepsilon)\|A x-b\|^2
-(1+\varepsilon)\|\nabla\Phi(x)+A^\top\lambda\|^2\\
&\quad+\varepsilon\Big(2\eta_1\|A x-b\|^2
+2(\eta_2+\eta_3)\|\nabla\Phi(x)+A^\top\lambda\|^2\Big)\\
&\quad+\varepsilon\Big(
\frac{\|A\|^2}{8\eta_1}\|x-x^*\|^2
+\frac{\|A\|^2}{8\eta_2}\|\lambda-\lambda^*\|^2
+\frac{L^2}{8\eta_3}\|x-x^*\|^2\Big).
\end{aligned}
\]
Collecting coefficients gives
\[
\dot{W}_\varepsilon
\le -\big(1+\varepsilon-2\varepsilon\eta_1\big)\|A x-b\|^2
-\big(1+\varepsilon-2\varepsilon(\eta_2+\eta_3)\big)\|\nabla\Phi(x)+A^\top\lambda\|^2
+\varepsilon R(x,\lambda),
\]
where
\[
R(x,\lambda)
=\frac{\|A\|^2}{8\eta_1}\|x-x^*\|^2
+\frac{\|A\|^2}{8\eta_2}\|\lambda-\lambda^*\|^2
+\frac{L^2}{8\eta_3}\|x-x^*\|^2.
\]
To control $R$ by $W$, note that since $x^*$ minimizes $\Phi$, we have $\Phi(x)-\Phi(x^*)\ge0$, and therefore $W(x,\lambda)\ge0$. Moreover,
\[
\|\lambda-\lambda^*\|^2\le 2\,W(x,\lambda), 
\qquad 
\|A x-b\|^2=\|A(x-x^*)\|^2\le \|A\|^2\|x-x^*\|^2.
\]
Working on a bounded sublevel $\{W\le W_0\}$, which is compact since the trajectory is bounded, there exist $\alpha'>0$ and $r>0$ such that
\[
\|x-x^*\|^2\le \frac{1}{\alpha'}\,W(x,\lambda)
\quad\text{for }\|(x,\lambda)-(x^*,\lambda^*)\|\le r.
\]
Outside this neighborhood, continuity on the compact set implies
\[
\delta:=\min_{K}W>0,
\qquad 
M:=\max_{\{W\le W_0\}}\|x-x^*\|<\infty,
\]
so that $\|x-x^*\|^2\le (M^2/\delta)\,W(x,\lambda)$ for all $(x,\lambda)\in K$. Setting
\[
C_x=\max\Big\{\frac{1}{\alpha'},\frac{M^2}{\delta}\Big\},
\]
we have $\|x-x^*\|^2\le C_x\,W(x,\lambda)$ on $\{W\le W_0\}$, and thus
\[
R(x,\lambda)\le
\Big(\frac{\|A\|^2}{8\eta_1}+\frac{L^2}{8\eta_3}\Big)C_x\,W(x,\lambda)
+\frac{\|A\|^2}{8\eta_2}\cdot 2\,W(x,\lambda)
= c_0\,W(x,\lambda),
\]
where
\[
c_0
=\Big(\frac{\|A\|^2}{8\eta_1}+\frac{L^2}{8\eta_3}\Big)C_x+\frac{\|A\|^2}{4\eta_2}.
\]
Hence,
\[
\dot{W}_\varepsilon
\le -a_1\|A x-b\|^2
-a_2\|\nabla\Phi(x)+A^\top\lambda\|^2
+\varepsilon c_0 W,
\qquad
\begin{cases}
a_1=1+\varepsilon-2\varepsilon\eta_1,\\[0.2em]
a_2=1+\varepsilon-2\varepsilon(\eta_2+\eta_3).
\end{cases}
\]
On the same compact set, there exists $\kappa>0$ such that
\[
W(x,\lambda)
\le \kappa\big(\|A x-b\|^2+\|\nabla\Phi(x)+A^\top\lambda\|^2\big),
\]
since both sides are continuous and vanish only at $(x^*,\lambda^*)$. Consequently,
\[
\varepsilon c_0 W
\le \varepsilon c_0 \kappa
\big(\|A x-b\|^2+\|\nabla\Phi(x)+A^\top\lambda\|^2\big).
\]
Choosing $\varepsilon>0$ sufficiently small so that 
$\varepsilon c_0 \kappa \le \tfrac12\min\{a_1,a_2\}$, we get
\[
\dot{W}_\varepsilon
\le -(a_1-\varepsilon c_0 \kappa)\|A x-b\|^2
-(a_2-\varepsilon c_0 \kappa)\|\nabla\Phi(x)+A^\top\lambda\|^2.
\]
Both coefficients are strictly positive; setting
\[
c_1=a_1-\varepsilon c_0\kappa>0,
\qquad
c_2=a_2-\varepsilon c_0\kappa>0,
\]
we finally obtain
\[
\dot{W}_\varepsilon
\le -c_1\|A x-b\|^2
-c_2\|\nabla\Phi(x)+A^\top\lambda\|^2.
\]
This shows that $\dot W_\varepsilon$ is strictly negative outside equilibrium along bounded trajectories.

\medskip
\noindent
To conclude, the perturbed Lyapunov function $W_\varepsilon$ is positive definite with respect to the KKT set $S$ and strictly decreases along trajectories outside $S$. Therefore every KKT equilibrium is Lyapunov stable, and
\[
\|A x(t)-b\|\to 0,
\qquad
\|\nabla\Phi(x(t))+A^\top\lambda(t)\|\to 0,
\]
whence
\[
\operatorname{dist}((x(t),\lambda(t)),S)\to 0.
\]
Thus $S$ is \emph{semistable}. If the constrained minimizer of $\Phi$ subject to $A x=b$ is unique, the trajectory converges to that single KKT point.
\section{Conclusion}

The paper developed a composite Lyapunov framework that establishes strict decay through a pair of differential inequalities.
This construction provides a unified and generalized structure that encompasses both the classical Lyapunov method—based on single-function decay—and the Matrosov auxiliary-function framework, where multiple functions interact to ensure convergence.
By merging these two perspectives into a single principle of composite decay inequalities, the approach delivers a direct and constructive path from decay estimates to integral bounds, convergence to the critical set, and stability of equilibria and invariant sets, all without relying on compactness or invariance assumptions.
Applications to inertial gradient–like systems and primal–dual gradient flows illustrate how this unified framework captures semistability and convergence in optimization-driven dynamics while offering a transparent analytical tool for broader classes of nonlinear systems.





\bibliographystyle{plain}
\bibliography{refs}

\begin{thebibliography}{10}

\bibitem{Attouch}
F.~Alvarez, H.~Attouch, J.~Bolte, and P.~Redont.
\newblock A second-order gradient-like dissipative dynamical system with hessian-driven damping: Application to optimization and mechanics.
\newblock {\em J. Math. Pures Appl. (9)}, 81(8):747--779, 2002.

\bibitem{Apidopoulos2025}
V.~Apidopoulos, C.~Molinari, J.~Peypouquet, and S.~Villa.
\newblock Preconditioned primal-dual dynamics in convex optimization: non-ergodic convergence rates.
\newblock {\em arXiv preprint arXiv:2506.00501}, 2025.

\bibitem{Arrow1958}
K.~J. Arrow, L.~Hurwicz, and H.~Uzawa.
\newblock {\em Studies in Linear and Nonlinear Programming}.
\newblock Stanford University Press, Stanford, California, 1958.

\bibitem{AstolfiPraly2011}
A.~Astolfi and L.~Praly.
\newblock A {LaSalle} version of {Matrosov}'s theorem.
\newblock In {\em Proceedings of the 50th IEEE Conference on Decision and Control and European Control Conference}, pages 320--324, 2011.

\bibitem{bhat99}
S.~P. Bhat and D.~S. Bernstein.
\newblock Lyapunov analysis of semistability.
\newblock In {\em Proceedings of the 1999 American Control Conference (Cat. No. 99CH36251)}, volume~3, pages 1608--1612, 1999.

\bibitem{Cherukuri2016}
A.~Cherukuri, E.~Mallada, and J.~Cort{\'e}s.
\newblock Asymptotic convergence of constrained primal--dual dynamics.
\newblock {\em Systems \& Control Letters}, 87:10--15, 2016.

\bibitem{Daher2012}
W.~Daher, L.~J. Mirman, and M.~Santugini.
\newblock Information in cournot: Signaling with incomplete control.
\newblock {\em International Journal of Industrial Organization}, 30(4):361--370, 2012.

\bibitem{dao2023}
M.~N. Dao, H.~Saoud, and M.~Th{\'e}ra.
\newblock Locating theorems of differential inclusions governed by maximally monotone operators.
\newblock {\em SIAM J. Optim.}, 33(4):2703--2720, 2023.

\bibitem{Feijer2010}
D.~Feijer and F.~Paganini.
\newblock Stability of primal--dual gradient dynamics and applications to network optimization.
\newblock {\em Automatica}, 46(12):1974--1981, 2010.

\bibitem{goebel10}
R.~Goebel.
\newblock Basic results on pointwise asymptotic stability and set-valued {Lyapunov} functions.
\newblock In {\em 49th IEEE Conference on Decision and Control (CDC)}, pages 1571--1574, 2010.

\bibitem{goebel16}
R.~Goebel and R.~G. Sanfelice.
\newblock Notions and sufficient conditions for pointwise asymptotic stability in hybrid systems.
\newblock {\em IFAC-PapersOnLine}, 49(18):140--145, 2016.

\bibitem{goebel18}
R.~Goebel and R.~G. Sanfelice.
\newblock Pointwise asymptotic stability in a hybrid system and well-posed behavior beyond zeno.
\newblock {\em SIAM J. Control Optim.}, 56(2):1358--1385, 2018.

\bibitem{hui09}
Q.~Hui, W.~M. Haddad, and S.~P. Bhat.
\newblock Semistability, finite-time stability, differential inclusions, and discontinuous dynamical systems having a continuum of equilibria.
\newblock {\em IEEE Trans. Autom. Control}, 54(10):2465--2470, 2009.

\bibitem{khalil}
H.~K. Khalil.
\newblock {\em Nonlinear Systems}.
\newblock Pearson Education. Prentice Hall, Upper Saddle River, NJ, 3 edition, 2002.

\bibitem{Loria2005}
A.~Loria, E.~Panteley, D.~Popovic, and A.~R. Teel.
\newblock A nested matrosov theorem and persistency of excitation for uniform convergence in stable nonautonomous systems.
\newblock {\em IEEE Transactions on Automatic Control}, 50(2):183--198, 2005.

\bibitem{Matrosov1962}
V.~M. Matrosov.
\newblock On the stability of motion.
\newblock {\em Journal of Applied Mathematics and Mechanics}, 26(5):1337--1353, 1962.

\bibitem{Mazenc2007}
F.~Mazenc and D.~Ne{\v{s}}i{\'c}.
\newblock Lyapunov functions for time--varying systems satisfying generalized conditions of the matrosov theorem.
\newblock {\em Mathematics of Control, Signals, and Systems}, 19(2):151--182, 2007.

\bibitem{Ozaslan2024}
Ibrahim~K Ozaslan, Panagiotis Patrinos, and Mihailo~R Jovanovi{\'c}.
\newblock Stability of primal-dual gradient flow dynamics for multi-block convex optimization problems.
\newblock {\em arXiv preprint arXiv:2408.15969}, 2024.

\bibitem{saoud15}
H.~Saoud.
\newblock Semistability of first-order evolution variational inequalities.
\newblock {\em Electron. J. Differ. Equ.}, 2015(265):1--10, 2015.

\bibitem{SaoudTheraDao2025}
Hassan Saoud, Michel Th{\'e}ra, and Minh~N Dao.
\newblock Geometric stability analysis for differential inclusions governed by maximally monotone operators.
\newblock {\em arXiv preprint arXiv:2507.13000}, 2025.

\bibitem{Teel2016}
A.~R. Teel, D.~Ne{\v{s}}i{\'c}, T.-C. Lee, and Y.~Tan.
\newblock A refinement of {Matrosov}'s theorem for differential inclusions.
\newblock {\em Automatica}, 68:378--383, 2016.

\end{thebibliography}

\end{document}